\documentclass[a4paper,10pt]{article}
\usepackage[utf8x]{inputenc}
\usepackage{tracefnt,amsmath,tabu,array}
\usepackage{amssymb,graphicx,setspace,amsfonts,amsbsy}
\usepackage{pifont,latexsym,ifthen,amsthm,rotating,calc,textcase,booktabs}

\newtheorem{theorem}{Theorem}[section]
\newtheorem{lemma}[theorem]{Lemma}

\newtheorem{proposition}[theorem]{Proposition}
\newtheorem{definition}[theorem]{Definition}

\newtheorem{remark}[theorem]{Remark}
\newcommand{\filledbox}{\leavevmode
  \hbox to.77778em{%
  \hfil\vbox to.675em{\hrule width.6em height.6em}\hfil}}

\newcommand{\EE}{{\mathcal E}}
\newcommand{\Rm}{{\mathbb R}}
\newcommand{\Hm}{{\mathbb H}}
\newcommand{\Db}{{\mathbf D}}
\newcommand{\eps}{\varepsilon}

\begin{document}
%\doublespacing
\tabulinesep=1.0mm

\title{Energy-critical semi-linear shifted wave equation on the hyperbolic spaces\footnote{MSC classes: 35L71, 35L05}}

\author{Ruipeng Shen\\
Center for Applied Mathematics\\
Tianjin University\\
Tianjin, P.R.China}

\maketitle

\begin{abstract}
  In this paper we consider a semi-linear, energy-critical, shifted wave equation on the hyperbolic space ${\mathbb H}^n$ with $3 \leq n \leq 5$:
  \[
  \partial_t^2 u - (\Delta_{{\mathbb H}^n} + \rho^2) u = \zeta |u|^{4/(n-2)} u, \quad (x,t)\in {\mathbb H}^n \times {\mathbb R}.
  \]
 Here $\zeta = \pm 1$ and $\rho = (n-1)/2$ are constants. We introduce a family of Strichartz estimates compatible with initial data in the energy space $H^{0,1} \times L^2 ({\mathbb H}^n)$ and then establish a local theory with these initial data. In addition, we prove a Morawetz-type inequality
  \[
   \int_{-T_-}^{T_+} \int_{{\mathbb H}^n} \frac{\rho (\cosh |x|) |u(x,t)|^{2n/(n-2)}}{\sinh |x|} d\mu(x) dt \leq n {\mathcal E},
  \]
 in the defocusing case $\zeta = -1$, where ${\mathcal E}$ is the energy. Moreover, if the initial data are also radial, we can prove the scattering of the corresponding solutions by combining the Morawetz-type inequality, the local theory and a pointwise estimate on radial $H^{0,1}({\mathbb H}^n)$ functions.
\end{abstract}

\section{Introduction}

In this work we continue our discussion on a semi-linear shifted wave equation on $\Hm^n$:
\begin{equation}\label{(CP1)}
\left\{\begin{array}{l} \partial_t^2 u - (\Delta_{\Hm^n} + \rho^2) u = \zeta |u|^{p-1} u,\quad (x,t)\in \Hm^n \times \Rm;\\
u|_{t=0} = u_0; \\
\partial_t u|_{t=0} = u_1.\end{array}\right.
\end{equation}
Here the constants satisfy $\rho=(n-1)/2$, $\zeta = \pm 1$ and $p>1$. We call this equation defocusing if $\zeta = -1$, otherwise if $\zeta=1$ we call it focusing.  The energy-subcritical case ($p < p_c := 1 + 4/(n-2)$, $2\leq n\leq 6$) has been considered by the author's recent joint work with Staffilani \cite{subhyper}. As a continuation, this work is concerned with the energy critical case $p = p_c$, $3\leq n\leq 5$. 

\paragraph{An Analogue of the wave equation in $\Rm^n$}  The equation \eqref{(CP1)} discussed in this work is the $\Hm^n$ analogue of the semi-linear wave equation defined in Euclidean space $\Rm^n$:
\begin{equation*} 
 \partial_t^2 u - \Delta u = \zeta |u|^{p-1}u,\quad (x,t)\in \Rm^n \times \Rm.
\end{equation*}
This similarity can be understood in two different ways, as we have already mentioned in \cite{subhyper}.
\begin{itemize}
  \item [(I)] The operator $-\Delta_{\Hm^n}-\rho^2$ in the hyperbolic space and the Laplace operator $-\Delta$ in $\Rm^n$ share the same Fourier symbol $\lambda^2$, as mentioned in Definition \ref{Sobo} below.
  \item [(II)] There is a transformation between solutions of the linear wave equation defined in a forward light cone in $\Rm^n \times \Rm$ and solutions of the linear shifted wave equation defined in the whole space-time $\Hm^n \times \Rm$. Please see Tataru \cite{tataru} for more details.
\end{itemize}
The author would also like to mention one major difference between these two equations. The symmetric group of the solutions to $\Rm^n$ wave equation includes the natural dilations $(\mathbf{T}_\lambda u)(x,t) \doteq \lambda^{-2/(p-1)} u(x/\lambda, t/\lambda)$, where $\lambda$ is an arbitrary positive constant. The shifted wave equation \eqref{(CP1)} on the hyperbolic spaces, however, does not possess a similar property of dilation-invariance.

\paragraph{The energy} Suitable  solutions to \eqref{(CP1)} satisfy the energy conservation law:
\begin{equation*} 
 \EE(u, u_t) = \int_{\Hm^n} \left[\frac{1}{2}(|\nabla u|^2 - \rho^2 |u|^2) + \frac{1}{2}|\partial_t u|^2 - \frac{\zeta}{p+1}|u|^{p+1} \right] d\mu = \hbox{const},
\end{equation*}
where $d \mu$ is the volume element on $\Hm^n$. Since the spectrum of $-\Delta_{\Hm^n}$ is $[\rho^2, \infty)$, it follows that the integral of $|\nabla u|^2 - \rho^2 |u|^2$ above is always nonnegative. We can also rewrite the energy in terms of certain norms. (Please see Definition \ref{Sobo} for the definition of $\dot{H}^{0,1}$ norm)
\[
 \EE = \frac{1}{2}\|u\|_{H^{0,1}(\Hm^n)}^2 + \frac{1}{2}\|\partial_t u\|_{L^2 (\Hm^n)}^2 - \frac{\zeta}{p+1} \|u\|_{L^{p+1}(\Hm^n)}^{p+1}.
\]
Please note that a solution in the focusing case may come with a negative energy. 

\paragraph{Previous results on $\Rm^n$} Let us first recall a few results regarding the energy-critical wave equation on Euclidean spaces. In 1990's Grillakis \cite{mg1, mg2} and Shatah-Struwe \cite{ss1,ss2} proved in the defocusing case that the solutions with any $\dot{H}^1 \times L^2$ initial data exist globally in time and scatter. The focusing case is more subtle and has been the subject of many more recent works such as Kenig-Merle \cite{kenig} (dimension $3\leq n \leq 5$, energy below the ground state), Duyckaerts-Kenig-Merle \cite{secret, tkm1} (radial case in dimension 3), Krieger-Nakanishi-Schlag \cite{radial dynamics, nonradial dynamics} (energy slightly above the ground state). 

\paragraph{Previous results on $\Hm^n$} Much less has been known in the case of hyperbolic spaces. Strichartz-type estimates for shifted wave equations have been discussed by Tataru \cite{tataru} and Ionesco \cite{SThyper}. More recently Anker, Pierfelice and Vallarino gave a wider range of Strichartz estimates and a brief description on the local well-posedness theory for the energy-subcritical case ($p<p_c$) in their work \cite{wavehyper}. %Global well-posedness is also considered in \cite{wavedr}. 
The author's joint work with Staffilani \cite{subhyper} improved their local theory and proved the global existence and scattering of solutions in the defocusing case with any $H^{1/2,1/2} \times H^{1/2,-1/2} (\Hm^n)$ initial data using a Morawetz type inequality, if $2 \leq n \leq 6$. Finally some global existence and scattering results have also been proved by A. French \cite{superhyp} in the energy-supercritical case, but only for small initial data.

\paragraph{Goal and main idea of this paper} This paper is divided into two parts. The first part is concerned with the local theory of the energy-critical shifted wave equation in hyperbolic spaces $\Hm^n$ with $3 \leq n \leq 5$:
\begin{equation*}
\left\{\begin{array}{l} \partial_t^2 u - (\Delta_{\Hm^n} + \rho^2) u = \zeta |u|^{4/(n-2)} u,\quad (x,t)\in \Hm^n \times \Rm;\\
u|_{t=0} = u_0; \\
\partial_t u|_{t=0} = u_1.\end{array}\right.\qquad \qquad \qquad (CP1)
\end{equation*}
We will first introduce a family of new Strichartz estimates via a $T T^\star$ argument and then establish a local well-posedness theory for any initial data in the energy space $H^{0,1} \times L^2 (\Hm^n)$.
The second part is about the global behaviour of solutions in the defocusing case. We will prove a Morawetz-type inequality
\begin{equation} \label{Morawetz}
        \int_{-T_-}^{T_+} \int_{\Hm^n} \frac{\rho (\cosh |x|) |u(x,t)|^{2n/(n-2)}}{\sinh |x|} d\mu(x) dt \leq n \EE.
\end{equation}
As in the Euclidean spaces, global space-time integral estimates of this kind are a powerful tool to discuss global behaviour of solutions. Although we are still not able to show the scattering of solutions with arbitrary initial data in $H^{0,1} \times L^2(\Hm^n)$, which we expect to be true, the Strichartz estimate above is sufficient to prove the scattering in the radial case, thanks to a point-wise estimate on radial $H^{0,1}(\Hm^n)$ functions as given in Lemma \ref{radialest}.

\paragraph{Main Results} For the convenience of readers, we briefly describe our main results as follows. We always assume that $3 \leq n \leq 5$ in this paper.
\begin{itemize}
  \item[(I)] For any initial data $(u_0,u_1) \in H^{0,1} \times L^2 (\Hm^n)$, there exists a unique solution to the equation (CP1) in a maximal time interval $(-T_-,T_+)$. 
  \item[(II)] In addition, if $u$ is a solution to (CP1) in the defocusing case with initial data $(u_0,u_1) \in H^{0,1}\times L^2 (\Hm^n)$, then it satisfies the Morawetz-type estimate \eqref{Morawetz}.
  \item[(III)] Moreover, if the initial data $(u_0,u_1) \in H^{0,1}\times L^2 (\Hm^n)$ are radial, then the solution to the equation (CP1) in the defocusing case exists globally in time and scatters. It is equivalent to saying that the maximal lifespan of the solution $u$ is $\Rm$ and there exist two pairs $(u_0^{\pm}, u_1^{\pm}) \in H^{0,1} \times L^2 (\Hm^n)$, such that
      \[
       \lim_{t \rightarrow \pm \infty} \left\|\left(u(\cdot, t), \partial_t u (\cdot,t)\right) - \mathbf{S}_L (t)(u_0^\pm, u_1^\pm)\right\|_{H^{0, 1}\times L^2(\Hm^n)} = 0.
      \]
      Here $\mathbf{S}_L (t)$ is the linear propagation operator for the shifted wave equation on $\Hm^n$ as defined in Section \ref{sec: notation}. 
\end{itemize}

\section{Notations and Preliminary Results}

\subsection{Notations} \label{sec: notation}

\paragraph{The notation $\lesssim$} We use the notation $A \lesssim B$ if there exists a constant $c$ such that $A \leq cB$. 

\paragraph{Linear propagation operator} Given a pair of initial data $(u_0, u_1)$, we use the notation $\mathbf{S}_{L,0}(t) (u_0,u_1)$ to represent the solution $u$ of the free linear shifted wave equation $\partial_t^2 u - (\Delta_{\Hm^n}+\rho^2)u = 0$ with initial data $(u, \partial_t u)|_{t=0} = (u_0,u_1)$. If we are also interested in the velocity $\partial_t u$, we can use the notations 
\begin{align*}
 &\mathbf{S}_L (t) (u_0,u_1) \doteq (u(\cdot, t), \partial_t u (\cdot, t)),& &\mathbf{S}_L (t) \begin{pmatrix} u_0\\ u_1 \end{pmatrix} \doteq \begin{pmatrix} u(\cdot, t)\\ \partial_t u(\cdot, t) \end{pmatrix}.& 
\end{align*}

\subsection{Fourier Analysis}

In order to make this paper self-contained, we make a brief review on the basic knowledge of the hyperbolic spaces and the related Fourier analysis in this subsection. 

\paragraph{Model of hyperbolic space} We use the hyperboloid model for hyperbolic space $\Hm^n$ in this paper. We start by considering Minkowswi space $\Rm^{n+1}$ equipped with the standard Minkowswi metric $-(dx^0)^2 + (dx^1)^2 + \cdots + (dx^n)^2$ and the bilinear form $[x,y] = x_0 y_0 - x_1 y_1 - \cdots - x_n y_n$. The hyperbolic space $\Hm^n$ can be defined as the upper sheet of the hyperboloid $x_0^2 - x_1^2 -\cdots -x_n^2 =1$. The Minkowswi metric then induces the metric, covariant derivatives $\mathbf D$ and measure $d\mu$ on the hyperbolic space $\Hm^n$.

\paragraph{Fourier transform} (Please see \cite{fourier1, fourier2} for more details) The Fourier transform takes suitable functions defined on $\Hm^n$ to functions defined on $(\lambda, \omega) \in \Rm \times {\mathbb S}^{n-1}$. We can write down the Fourier transform of a function $f\in C_0^{\infty} (\Hm^n)$ and the inverse Fourier transform by
\begin{align*}
 \tilde{f} (\lambda,\omega) &= \int_{\Hm^n} f(x) [x,b(\omega)]^{i\lambda -\rho} d\mu(x);\\
 f(x) & = \hbox{const.} \int_0^\infty \int_{{\mathbb S}^{n-1}} \tilde{f}(\lambda,\omega) [x,b(\omega)]^{-i\lambda -\rho} |\mathbf{c}(\lambda)|^{-2} d\omega d\lambda;
\end{align*}
where $b(\omega)$ and the Harish-Chandra $\mathbf{c}$-function $\mathbf{c}(\lambda)$ are defined by ($C_n$ is a constant determined solely by the dimension $n$)
\begin{align*}
 &b(\omega) = (1, \omega) \in \Rm^{n+1};& &\mathbf{c}(\lambda) = %\frac{\Gamma(2 \rho)}{\Gamma (\rho)} \cdot 
 C_n \frac{\Gamma (i \lambda)}{\Gamma (i \lambda + \rho)}.&
\end{align*}
It is well known that $|\mathbf{c}(\lambda)|^{-2} \lesssim |\lambda|^2 (1+|\lambda|)^{n-3}$. The Fourier transform $f \rightarrow \tilde{f}$ defined above extends to an isometry from $L^2(\Hm^n)$ onto $L^2 (\Rm^+\times {\mathbb S}^{n-1}, |\mathbf{c}(\lambda)|^{-2} d\lambda d\omega)$ with the Plancheral identity:
\[
 \int_{\Hm^n} f_1 (x) \overline{f_2 (x)} dx = \int_0^\infty \int_{{\mathbb S}^{n-1}} \tilde{f}_1(\lambda,\omega) \overline{\tilde{f}_2 (\lambda, \omega)} |\mathbf{c}(\lambda)|^{-2} d\omega d\lambda.
\]
We also have an identity $\widetilde{-\Delta_{\Hm^n} f} = (\lambda^2 + \rho^2) \tilde{f}$ for the Laplace operator $\Delta_{\Hm^n}$.

\paragraph{Radial Functions} Let us use the polar coordinates $(r, \Theta) \in [0,\infty) \times {\mathbb S}^{n-1}$ to represent the point $(\cosh r, \Theta \sinh r) \in \Hm^n \hookrightarrow \Rm^{n+1}$ in the hyperboloid model above. In particular, the $r$ coordinate of a point in $\Hm^n$ represents the metric distance from that point to the ``origin'' $\mathbf{0} \in \Hm^n$, which corresponds to the point $(1,0,\cdots,0)$ in the Minkiwski space. As in Euclidean spaces, for any $x\in \Hm^n$ we also use the notation $|x|$ for the same distance from $x$ to $\mathbf{0}$. Namely
\[
 r = |x| = d (x, \mathbf{0}), \qquad x \in \Hm^n.
\]
A function $f$ defined on $\Hm^n$ is radial if and only if it is independent of $\Theta$. By convention we may use the notation $f(r)$ for a radial function $f$. If the function $f(x)$ in question is radial, we can rewrite the Fourier transform and its inverse in a simpler form
\begin{align*}
 \tilde{f}(\lambda) &= \tilde{f}(\lambda, \omega) = \int_{\Hm^n} f(x) \Phi_{-\lambda} (x) d\mu(x);\\
 f(x) &= \hbox{const.} \int_0^\infty \tilde{f}(\lambda) \Phi_{\lambda} (x) |\mathbf{c}(\lambda)|^{-2} d\lambda.
\end{align*}
Here the function $\Phi_\lambda (x)$ is the elementary spherical (radial) function of $x \in \Hm^n$ defined by
\[
 \Phi_\lambda (x) = \int_{{\mathbb S}^{n-1}} [x, b(\omega)]^{-i \lambda-\rho} d\omega.
\]
One can use spherical coordinates on ${\mathbb S}^{n-1}$ to evaluate the integral and rewrite $\Phi_\lambda (x)$ into
\begin{equation} \label{formulaphi1}
 \Phi_\lambda (x) = \Phi_\lambda (r) = \frac{\Gamma(\frac{n}{2})}{\sqrt{\pi}\Gamma(\frac{n-1}{2})} \int_0^\pi (\cosh r - \sinh r \cos \theta)^{-i \lambda -\rho} \sin^{n-2} \theta \,d\theta.
\end{equation}
The change of variables $u = \ln (\cosh r - \sinh r \cos \theta)$ then gives another formula of $\Phi_\lambda (r)$ if $r > 0$:
\begin{equation*} 
\Phi_\lambda (r) = \frac{2^{\frac{n-3}{2}}\Gamma(\frac{n}{2})}{\sqrt{\pi}\Gamma(\frac{n-1}{2})} (\sinh r)^{2-n} \int_{-r}^r (\cosh r - \cosh u)^{\frac{n-3}{2}} e^{-i\lambda u}\, du.
\end{equation*}
These integral representations imply that
\begin{itemize}
  \item The function $\Phi_\lambda (r)$ is a real-valued function for all $r \geq 0$ and $\lambda \in \Rm$.
  \item The function $\Phi_\lambda (r)$ has an upper bound independent of $\lambda$
  \begin{equation} \label{universalphi}
   |\Phi_\lambda (r)| \leq \Phi_0 (r) \leq C e^{-\rho r}(r + 1).
  \end{equation}
\end{itemize}
In the 3-dimensional case, the function $\Phi_\lambda (r)$ is particularly easy and can be given by an explicit formula $\Phi_\lambda (r) = (\sin \lambda r)/(\lambda \sinh r)$.

\paragraph{Convolution} If $f, K \in C_0 (\Hm^n)$ and $K$ is radial, we can define the convolution $f \ast K$ by an integral 
\[
 (f \ast K)(x) = \int_{{\mathbb G}} f(g \cdot \mathbf{0}) K (g^{-1} \cdot x) dg.
\]
Here ${\mathbb G} = SO(1,n)$ is the connected Lie Group of $(n+1) \times (n+1)$ matrices that leave the bilinear form $[x, y]=x_0 y_0 - x_1 y_1 -\cdots -x_n y_n$ invariant. The notations $g\cdot \mathbf{0}$ and $g^{-1} \cdot x$ represent the natural action of $\mathbb G$ on $\Hm^n$ defined by the usual left-multiplication of matrices on vectors. The measure $dg$ is the Haar measure on $G$ normalized in such a way that the identity $\int_{{\mathbb G}} f(g\cdot \mathbf{0}) d g = \int_{\Hm^n} f(x) d\mu$ holds for any $f \in C_0 (\Hm^n)$. The Fourier transform of $f \ast K$ satisfies the identity
\[
 \left(\widetilde{f \ast K}\right) (\lambda, \omega) = \tilde{f} (\lambda, \omega) \cdot \tilde{K} (\lambda).
\]
The Fourier transform $\tilde{K}$ does not depend on $\omega$ since we have assumed that $K$ is radial. The author would like to emphasize that there is no simple identity of this type without radial assumption on $K$.  Please see \cite{hypersdg} for more details.

\subsection{Sobolev Spaces}

\begin{definition} \label{Sobo}
Let $D^\gamma = (-\Delta_{\Hm^n} -\rho^2)^{\gamma/2}$ and $\tilde{D}^\sigma = (-\Delta_{\Hm^n} +1)^{\sigma/2}$.  These operators can also be defined by Fourier multipliers $m_1(\lambda) =|\lambda|^\gamma$ and $m_2(\lambda) = (\lambda^2 + \rho^2 +1)^{\sigma/2}$, respectively. We define the following Sobolev spaces and norms for $\gamma < 3/2$.
\begin{align*}
 &H_q^{\sigma}(\Hm^n) = \tilde{D}^{-\sigma} L^q (\Hm^n),&
 &\|u\|_{H_q^{\sigma}(\Hm^n)}= \|\tilde{D}^\sigma u\|_{L^q (\Hm^n)};&\\
 &H^{\sigma,\gamma}(\Hm^n) = \tilde{D}^{-\sigma} D^{-\gamma} L^2 (\Hm^n),&
 &\|u\|_{H^{\sigma,\gamma}(\Hm^n)}= \|D^{\gamma}\tilde{D}^\sigma u\|_{L^2 (\Hm^n)}.&
\end{align*}
\end{definition}
\begin{remark}
If $\sigma$ is a positive integer, one can also define the Sobolev spaces by the Riemannian structure. For example, we can first define the $W^{1,q}$ norm as
\[
 \|u\|_{W^{1,p}} = \left(\int_{\Hm^n} |\nabla u|^{q} d\mu\right)^{1/q}
\]
for suitable functions $u$ and then take the closure. Here $|\nabla u| = (\Db_\alpha u \Db^\alpha u)^{1/2}$ is defined by the covariant derivatives. It turns out that these two definitions are equivalent to each other if $1 < q < \infty$, see \cite{tataru}. In other words, we have $\|u\|_{H_q^\sigma} \simeq \|u\|_{W^{\sigma,q}}$. In particular, we can rewrite the definition of $\dot{H}^{0,1}$ norm into 
\[
 \|u\|_{\dot{H}^{0,1}(\Rm^n)}^2 = \int_{\Hm^n} \left(|\nabla u|^2 - \rho^2 |u|^2\right)\, d\mu. 
\]
\end{remark}

\begin{definition}
Let $I$ be a time interval. The space-time norm is defined by
\[
 \|u(x,t)\|_{L^q L^r (I \times \Hm^n)} = \left( \int_{I} \left(\int_{\Hm^n} |u(x,t)|^r d\mu\right)^{q/r} dt \right)^{1/q}.
\]
\end{definition}
\begin{proposition} [Sobolev embedding]  Assume $1 < q_1 \leq q_2 < \infty$ and $\sigma_1, \sigma_2 \in \Rm$. If $\sigma_1 -\frac{n}{q_1} \geq \sigma_2 - \frac{n}{q_2}$, then we have the Sobolev embedding $H_{q_1}^{\sigma_1}(\Hm^n) \hookrightarrow H_{q_2}^{\sigma_2}(\Hm^n)$.
\end{proposition}
\noindent For the proof see \cite{wavehyper, sob} and the references cited therein.
\begin{proposition} \label{Sobolevem1}
(See Proposition 2.5 in \cite{subhyper}) If $q >2$, $0 < \tau < \frac{3}{2}$ and $\sigma + \tau \geq \frac{n}{2} - \frac{n}{q}$, then we have the Sobolev embedding $H^{\sigma, \tau}(\Hm^n) \hookrightarrow L^{q}(\Hm^n)$.
\end{proposition}

\subsection{Technical Lemma} \label{sec: estiphi}

In this subsection we introduce a point-wise estimate for radial $H^{0,1}(\Hm^n)$ functions. 

\begin{lemma} \label{radialest}
Let $n \geq 3$. We have a point-wise estimate $|f(r)| \lesssim_n r^{1/2}(\sinh r)^{-\rho} \|f\|_{H^{0,1}(\Hm^n)}$ for any radial function $f \in \dot{H}^{0,1}(\Hm^n)$. 
\end{lemma}
\begin{proof} 
Without loss of generality we assume that $f$ is smooth and has compact support, since functions of this kind are dense in the space of radial $\dot{H}^{0,1} (\Hm^n)$ functions. Let us first pick up a large radius $R$ so that $\hbox{Supp}(f) \subset B(\mathbf{0}, R)$ and then calculate 
\begin{align*}
 & \int_0^R \left[\frac{d}{dr}\left(f \sinh^\rho r \right)\right]^2\, dr\\
  = & \int_0^R \left[f_r \sinh^\rho r + \rho f \cdot (\sinh r)^{\rho-1}  \cosh r\right]^2\, dr\\
 = & \int_0^R \left[(f_r)^2 \sinh^{2\rho} r + \rho^2 f^2\cdot (\sinh r)^{2(\rho-1)} \cosh^2 r\right]\, dr + \int_0^R \rho (\sinh r)^{2\rho -1} \cosh r\, d(f^2)\\
 = & \int_0^R \left[(f_r)^2 \sinh^{2\rho} r + \rho^2 f^2 \cdot(\sinh r)^{2(\rho-1)} \cosh^2 r\right]\, dr - \int_0^R \rho f^2 \cdot \frac{d}{dr}\left[(\sinh r)^{2\rho-1} \cosh r\right]\, dr\\
 = & \int_0^R \left[(f_r)^2 - \rho^2 f^2\right] \sinh^{2\rho} r\, dr + (\rho - \rho^2) \int_0^R f^2 (\sinh r)^{2\rho-2} \, dr \\
 \lesssim_n & \int_{\Hm^n} [|\nabla f|^2 - \rho^2 |f|^2] d\mu  = \|f\|_{\dot{H}^{0,1} (\Hm^n)}^2. 
\end{align*}
As a result, we have 
\begin{align*}
 f(r') \sinh^\rho r' & = \int_0^{r'} \left[\frac{d}{dr}\left(f \sinh^\rho r \right)\right]\, dr
 \leq (r')^{1/2} \left\{\int_0^{r'} \left[\frac{d}{dr}\left(f \sinh^\rho r \right)\right]^2\, dr\right\}^{1/2}\\
 & \lesssim_n (r')^{1/2} \|f\|_{\dot{H}^{0,1} (\Hm^n)}
\end{align*}
and finish the proof. 
\end{proof}
\begin{remark} The upper bound given in Lemma \ref{radialest} is optimal. Given a smooth cut-off function $\varphi: \Rm \rightarrow [0,1]$ satisfying
\[
 \varphi (r) = \left\{\begin{array}{ll} 1, & \hbox{if $1/2 \leq r \leq 3/2$;}\\
 0, & \hbox{if $r<1/4$ or $r>7/4$;}  \end{array}\right.
\]
we consider a family of radial functions defined in $\Hm^n$:
\[
 f_R (r) = \left\{\begin{array}{ll} r^{1/2 -\rho} \varphi (r/R), & \hbox{if $R \leq 1$;}\\
 e^{-\rho r} r^{1/2} \varphi(r/R), &\hbox{if $R > 1$.}  \end{array}\right.
\]
One can check that $\|f_R (r)\|_{H^{0,1}(\Hm^n)} \lesssim 1$ and $f_R (R) \simeq R^{1/2} (\sinh R)^{-\rho}$.
\end{remark}

\section{Strichartz Estimates}

In this section we introduce a family of Strichartz estimates compatible with initial data in the energy space $H^{0,1}(\Hm^n) \times L^2 (\Hm^n)$ for all dimensions $n \geq 3$. This immediately leads to a local theory when $3 \leq n \leq 5$, which will be introduced in the next section.

\subsection{Preliminary Results}

\begin{definition} \label{originaladmissible}
 Let $n \geq 3$. A couple $(p_1, q_1)$ is called admissible if $(1/p_1, 1/q_1)$ belongs to the set
 \[
   T_n = \left\{\left(\frac{1}{p_1}, \frac{1}{q_1}\right) \in \left(0,\frac{1}{2}\right] \times \left(0, \frac{1}{2}\right) \left| \frac{2}{p_1} + \frac{n-1}{q_1} \geq \frac{n-1}{2}\right.\right\}.
 \]
\end{definition}
\noindent Let us recall the Strichartz estimates with inhomogeneous Sobolev norms. 
\begin{theorem} \label{originalstri} (See Theorem 6.3 in \cite{wavehyper})
 Let $(p_1,q_1)$ and $(p_2, q_2)$ be two admissible pairs. The real numbers $\sigma_1$ and $\sigma_2$ satisfy
 \begin{align*}
  &\sigma_1 \geq \beta (q_1) = \frac{n+1}{2} \left(\frac{1}{2} - \frac{1}{q_1}\right);&
  &\sigma_2 \geq \beta (q_2) = \frac{n+1}{2} \left(\frac{1}{2} - \frac{1}{q_2}\right).&
 \end{align*}
Assume $u(x,t)$ is the solution to the linear shifted wave equation
\begin{equation} \label{linearequation}
 \left\{\begin{array}{l} \partial_t^2 u - (\Delta_{\Hm^n}+\rho^2) u = F(x,t), \,\,\,\, (x,t)\in \Hm^n \times I;\\
u |_{t=0} = u_0; \\
\partial_t u |_{t=0} = u_1.\end{array}\right.
\end{equation}
Here $I$ is an arbitrary time interval containing $0$. Then we have
\begin{align*}
 \|u\|_{L^{p_1} L^{q_1}(I \times \Hm^n)}&
  + \|(u, \partial_t u)\|_{C(I; H^{\sigma_1 -\frac{1}{2},\frac{1}{2}}\times H^{\sigma_1 -\frac{1}{2},-\frac{1}{2}}(\Hm^n))} \\
 & \leq C \left( \|(u_0, u_1)\|_{H^{\sigma_1 -\frac{1}{2},\frac{1}{2}}\times H^{\sigma_1 -\frac{1}{2},-\frac{1}{2}}(\Hm^n)} + \|F\|_{L^{p'_2} (I; H_{q'_2}^{\sigma_1 + \sigma_2 - 1}(\Hm^n))} \right).
\end{align*}
The constant $C$ above does not depend on the time interval $I$.
\end{theorem}
The following lemma (see lemma 5.1 in \cite{wavehyper}) plays an important role in the proof of the Strichartz estimates above. It is obtained by a complex interpolation and the Kunze-Stein phenomenon.
\begin{lemma} \label{ksphenomenon}
There exists a constant $C > 0$ such that, for every radial measurable function $\kappa$ on $\Hm^n$, every $2 \leq q, \tilde{q} < \infty$ and $f \in L^{\tilde{q}'}(\Hm^n)$, we have 
\[
 \|f \ast \kappa\|_{L^q} \leq C \|f\|_{L^{\tilde{q}'}}\left(\int_0^\infty (\sinh r)^{n-1} (\Phi_0(r))^\gamma |\kappa (r)|^Q dr \right)^{1/Q}.
\]
Here $\gamma = \frac{2\min \{q, \tilde{q}\}}{q + \tilde{q}}$ and $Q = \frac{q \tilde{q}}{q + \tilde{q}}$.
\end{lemma}
\noindent The following lemma comes from a basic Fourier analysis
\begin{lemma} \label{L1L2striinter}
If $F  \in L^1 L^2 (\Rm \times \Hm^n)$, then we have
\[
 \left\|\int_{-\infty}^{\infty} e^{\pm i s D} F(\cdot, s) \,ds\right\|_{L^2(\Hm^n)} \leq \|F\|_{L^1 L^2 (\Rm \times \Hm^n)}.
\]
\end{lemma}
\subsection{Strichartz Estimates for $H^{0,1} \times L^2$ data}

\begin{definition}
 We fix $\chi: \Rm \rightarrow [0,1]$ to be an even, smooth cut-off function so that
 \[
  \chi(r) = \left\{\begin{array}{ll} 1, & |r|<1;\\
  0, & |r|>3/2.
  \end{array}\right.
 \]
\end{definition} 

\begin{lemma} \label{Striinter1} Given any $2 < p ,q < \infty$ and $\sigma \in \Rm$ we have
\[
  \left\|\chi (D) D^{-1} \tilde{D}^{1 -\sigma} e^{\pm i t D} f\right\|_{L^p L^q(\Rm \times \Hm^n)} \lesssim \|f\|_{L^2 (\Hm^n)}.
\]
\end{lemma}
\begin{proof}
Consider the operator $\chi^2 (D) D^{-2} \tilde{D}^{2 - 2\sigma} e^{i t D}$ defined by the Fourier multiplier
\[
\lambda \rightarrow \chi^2 (\lambda) \lambda^{-2} (\lambda^2 + \rho^2 + 1)^{1 -\sigma} e^{i t \lambda}
\]
and its kernel
\[
 \kappa_t^{\sigma} (r) = \hbox{const.} \int_0^2 \chi^2 (\lambda) \lambda^{-2} (\lambda^2 + \rho^2 + 1)^{1 -\sigma} e^{i t \lambda} \Phi_\lambda (r) |\mathbf{c}(\lambda)|^{-2} d\lambda.
\]
If $|t|\leq 2$, we recall $\left|\mathbf{c} (\lambda)\right|^{-2} \lesssim \lambda^2(1+|\lambda|)^{n-3}$ and obtain
\begin{equation}
 |\kappa_t^{\sigma} (r)| \lesssim \int_0^2 \chi^2 (\lambda) \lambda^{-2} (\lambda^2 + \rho^2 + 1)^{1 -\sigma} |e^{i t \lambda}| \cdot \Phi_0 (r) |\mathbf{c}(\lambda)|^{-2} d\lambda \lesssim \Phi_0 (r).
\end{equation}
Now let us consider the other case $|t| > 2$. By the definition of $\mathbf{c}(\lambda)$ we have
\[
 |\mathbf{c}(\lambda)|^{-2} = |C_n|^{-2} \frac{|\Gamma(i\lambda +\rho)|^2}{|\Gamma(i\lambda)|^2} = |C_n|^{-2} \lambda^2 \frac{|\Gamma(i\lambda +\rho)|^2}{|\Gamma(i\lambda+1)|^2}.
\]
Thus we can rewrite the kernel $\kappa_t^\sigma$ into
\[
 \kappa_t^{\sigma} (r) = \int_0^2 a(\lambda) e^{i t \lambda} \Phi_\lambda (r) d\lambda.
\]
Here the function
\[
  a(\lambda) = \hbox{const.} \chi^2 (\lambda)(\lambda^2 + \rho^2 + 1)^{1 -\sigma} \frac{|\Gamma(i\lambda +\rho)|^2}{|\Gamma(i\lambda+1)|^2}
\]
is smooth in $\Rm$. In addition, the function $\Phi_\lambda (r)$ satisfies
\begin{align*}
 \Phi_\lambda(r) = & c_n \int_0^\pi (\cosh r - \sinh r \cos \theta)^{-i\lambda -\rho} (\sin \theta)^{n-2} d\theta;\\
 |\Phi_\lambda(r)| \leq & \Phi_0 (r) = c_n \int_0^\pi (\cosh r - \sinh r \cos \theta)^{-\rho} (\sin \theta)^{n-2} d\theta \lesssim e^{-\rho r}(r+1);\\
 \partial_\lambda \Phi_\lambda (r) = & -c_n i \int_0^\pi \left[(\cosh r - \sinh r \cos \theta)^{-i\lambda -\rho} \ln (\cosh r- \sinh r \cos\theta)\right] (\sin \theta)^{n-2} d\theta;\\
 |\partial_\lambda \Phi_\lambda (r)| \leq & c_n \int_0^\pi (\cosh r - \sinh r \cos \theta)^{-\rho} |\ln (\cosh r- \sinh r \cos\theta)| (\sin \theta)^{n-2} d\theta\\
  \leq & \left(\sup_{\theta \in [0,\pi]} |\ln (\cosh r- \sinh r \cos\theta)|\right) \Phi_0 (r) \lesssim e^{-\rho r} r(r+1).
\end{align*}
We apply integration by parts on $\kappa_t^\sigma$ and obtain
\begin{align*}
 \kappa_t^{\sigma} (r) = & \frac{1}{i t} \int_0^2 a(\lambda) \Phi_\lambda (r) d (e^{i t \lambda})\\
 = & \frac{1}{it}\left[a(2) \Phi_2 (r) e^{2it} - a(0) \Phi_0 (r) \right] - \frac{1}{i t} \int_0^2 \partial_\lambda \left[a (\lambda) \Phi_\lambda (r)\right] e^{i t \lambda} d\lambda\\
 = & \frac{i}{t} a(0) \Phi_0 (r) + \frac{i}{t} \int_0^2 \left[(\partial_\lambda a (\lambda)) \Phi_\lambda (r) + a (\lambda) (\partial_\lambda \Phi_\lambda (r)) \right] e^{i t \lambda} d\lambda.
\end{align*}
As a result we have 
\begin{align*}
 |\kappa_t^{\sigma} (r)| \lesssim & \frac{|a(0)|}{|t|} \Phi_0 (r) + \frac{1}{|t|} \int_0^2 \left[|\partial_\lambda a (\lambda)| |\Phi_\lambda (r)| + |a (\lambda)| |\partial_\lambda \Phi_\lambda (r)| \right] d\lambda\\
 \lesssim & |t|^{-1} e^{-\rho r}(r+1) + |t|^{-1} \int_0^2 \left( \Phi_0 (r) + |\partial_\lambda \Phi_\lambda (r)| \right) d\lambda\\
 \lesssim & |t|^{-1} e^{-\rho r}(r+1)^2.
\end{align*}
Now let us apply Lemma \ref{ksphenomenon} with kernel $\kappa_t^\sigma$ and $\tilde{q}=q > 2$. In this case $\gamma =1$ and $Q = q/2 > 1$. The integral in Lemma \ref{ksphenomenon} can be estimated by:
\begin{itemize}
  \item If $|t| \leq 2$, we have
  \begin{align*}
   \int_0^{+\infty} (\sinh r)^{n-1} (\Phi_0(r))^\gamma |\kappa_t^\sigma (r)|^Q dr \lesssim & \int_0^{+\infty} (\sinh r)^{2\rho} \Phi_0(r) |\Phi_0 (r)|^{q/2} dr\\
   \lesssim & \int_0^{+\infty} (\sinh r)^{2\rho} \left(e^{-\rho r}(r+1)\right)^{1 + q/2} dr \lesssim 1.
  \end{align*}
  \item If $|t| > 2$, we have
  \begin{align*}
   \int_0^{+\infty} (\sinh r)^{n-1} (\Phi_0(r))^\gamma |\kappa_t^\sigma (r)|^Q dr
  % \lesssim & \int_0^{+\infty} (\sinh r)^{2\rho} \Phi_0(r) |\kappa_t^\sigma (r)|^{q/2} dr\\
   \lesssim & \int_0^{+\infty} (\sinh r)^{2\rho} e^{-\rho r}(r+1) [|t|^{-1} e^{-\rho r} (r+1)^2]^{q/2} dr\\
   \lesssim & |t|^{-q/2} \int_0^{+\infty} (\sinh r)^{2\rho} e^{-(1+ q/2)\rho r}(r+1)^{q+1} dr\\
   \lesssim & |t|^{-q/2}.
  \end{align*}
\end{itemize}
According to Lemma \ref{ksphenomenon}, we immediately have ($q > 2$)
\begin{equation}\label{kernel estimate1}
 \left\|\chi^2 (D) D^{-2} \tilde{D}^{2 - 2\sigma} e^{\pm i t D}\right\|_{L^{q'} \rightarrow L^{q}} \lesssim \left\{\begin{array}{ll}
 1, & |t|\leq 2;\\
 |t|^{-1}, & |t| > 2.
 \end{array}\right.
\end{equation}
Let us consider the operators
\begin{align*}
 \mathbf{T} f & = \chi (D) D^{-1} \tilde{D}^{1 -\sigma} e^{\pm i t D} f;\\
 \mathbf{T}^* F & = \int_{-\infty}^\infty \chi (D) D^{-1} \tilde{D}^{1 -\sigma} e^{\mp i s D} F(\cdot, s) \,ds;\\
 \mathbf{T} \mathbf{T}^* F & = \int_{-\infty}^\infty \chi^2 (D) D^{-2} \tilde{D}^{2 -2\sigma} e^{\pm i(t-s) D} F(\cdot, s)\, ds;
\end{align*}
This is clear that $\mathbf{T}^*$ is an operator from $L^1 (\Rm, H^{1-\sigma,-1}(\Hm^n)) \cap L^{p'} L^{q'}$ to $L^2 (\Hm^n)$, and that $\mathbf{T}$ is an operator from $L^2 (\Hm^n)$ to $L^\infty (\Rm, H^{\sigma-1, 1}(\Hm^n))$. Furthermore, the estimate (\ref{kernel estimate1}) guarantees the inequality
\[
 \left\|\mathbf{T} \mathbf{T}^* F\right\|_{L^p L^q(\Rm \times \Hm^n)} \lesssim \|F\|_{L^{p'} L^{q'} (\Rm \times \Hm^n)}
\]
holds as long as $p > 2$. By the $\mathbf{T} \mathbf{T}^*$ argument (see \cite{strichartz}, for instance), we obtain
\[
 \left\|\chi (D) D^{-1} \tilde{D}^{1 -\sigma} e^{\pm i t D} f\right\|_{L^p L^q(\Rm \times \Hm^n)} = \|\mathbf{T} f\|_{L^p L^q(\Rm \times \Hm^n)} \lesssim \|f\|_{L^2(\Hm^n)}
\]
thus finish the proof.
\end{proof}
\begin{theorem} [Strichartz estimates for $H^{0,1} \times L^2$ initial data] \label{strichartzpre}
Let $n\geq 3$. If $2<p,q<\infty$ and $\sigma$ satisfy
\begin{align*}
  &\frac{2}{p} + \frac{n-1}{q} \geq \frac{n-1}{2};& &\sigma \geq \frac{n+1}{2} \left(\frac{1}{2} - \frac{1}{q}\right);&
\end{align*}
then there exists a constant $C$, so that the solution $u$ to linear shifted wave equation $\partial_t^2 u - \Delta_{\Hm^n} u = F$, $(x,t)\in \Hm^n \times I$ with initial data $(u_0,u_1)$ satisfies
\begin{align*}
 \|\tilde{D}^{1-\sigma} u\|_{L^{p} L^{q}(I \times \Hm^n)}&
  + \|(u, \partial_t u)\|_{C(I; H^{0,1}\times L^2(\Hm^n))} \\
 & \leq C \left( \|(u_0, u_1)\|_{H^{0,1}\times L^2(\Hm^n)} + \|F\|_{L^1 L^2 (I\times \Hm^n)} \right).
\end{align*} \label{StriH01a}
\end{theorem}

\begin{proof}
Without loss of generality, let us assume $I = \Rm$. We start with the free linear propagation $u_L$ with a pair of arbitrary initial data $(u_0,u_1)$. In fact we have
\begin{equation}\label{freep}
 u_L (\cdot,t) = \cos (t D) u_0 + \frac{\sin (t D)}{D} u_1.
\end{equation}
Since $\tilde{D}^{1-\sigma} u_L$ solves the free linear shifted wave equation with initial data $(\tilde{D}^{1-\sigma} u_0, \tilde{D}^{1-\sigma} u_1)$, Theorem \ref{originalstri} immediately gives
\[
 \|\tilde{D}^{1-\sigma} u_L\|_{L^{p} L^{q}(\Rm \times \Hm^n)} \lesssim \|\tilde{D}^{1-\sigma} (u_0, u_1)\|_{H^{\sigma -\frac{1}{2},\frac{1}{2}}\times H^{\sigma -\frac{1}{2},-\frac{1}{2}}(\Hm^n)} = \|(u_0, u_1)\|_{H^{\frac{1}{2},\frac{1}{2}}\times H^{\frac{1}{2},-\frac{1}{2}}(\Hm^n)}.
\]
We rewrite this in the form of operators by the identity \eqref{freep} and obtain 
\begin{equation} \label{operator2}
 \|D^{-1} \tilde{D}^{1 -\sigma} e^{\pm i t D} f\|_{L^{p} L^{q}(\Rm \times \Hm^n)} \lesssim \|f\|_{H^{\frac{1}{2}, -\frac{1}{2}}(\Hm^n)}.
\end{equation}
Given an arbitrary $f \in L^2 (\Hm^n)$, the combination of (\ref{operator2}) and Lemma \ref{Striinter1} gives
\begin{align}
 &\|D^{-1} \tilde{D}^{1 -\sigma} e^{\pm i t D} f\|_{L^{p} L^{q}(\Rm \times \Hm^n)}\nonumber\\
  \leq & \|\chi(D) D^{-1} \tilde{D}^{1 -\sigma} e^{\pm i t D} f\|_{L^{p} L^{q}} + \|D^{-1} \tilde{D}^{1 -\sigma} e^{\pm i t D} (1 - \chi(D)) f\|_{L^{p} L^{q}}\nonumber\\
 \lesssim & \|f\|_{L^2(\Hm^n)} + \|(1 -\chi(D))f\|_{H^{\frac{1}{2}, -\frac{1}{2}}(\Hm^n)}\nonumber\\
 \lesssim & \|f\|_{L^2(\Hm^n)}. \label{leftestimate}
\end{align}
We combine this with Lemma \ref{L1L2striinter} and obtain
\begin{align*}
\left\|\int_{-\infty}^{\infty} D^{-1} \tilde{D}^{1 -\sigma} e^{\pm i (t-s) D} F(\cdot,s) \,ds \right\|_{L^{p} L^{q}(\Rm \times \Hm^n)} &\lesssim \|F\|_{L^1 L^2(\Rm \times \Hm^n)};\\
 \left\|\int_{-\infty}^{0} D^{-1} \tilde{D}^{1 -\sigma} e^{\pm i (t-s) D} F(\cdot,s) \, ds \right\|_{L^{p} L^{q}(\Rm \times \Hm^n)} &\lesssim \|F\|_{L^1 L^2(\Rm \times \Hm^n)}.
\end{align*}
According to Theorem 1.1 in \cite{truncation}, we also have a truncated version of the first inequality above
\[
 \left\|\int_{-\infty}^{t} D^{-1} \tilde{D}^{1 -\sigma} e^{\pm i (t-s) D} F(\cdot,s) \,ds \right\|_{L^{p} L^{q}(\Rm \times \Hm^n)} \lesssim \|F\|_{L^1 L^2(\Rm \times \Hm^n)}.
\]
Therefore we have 
\begin{equation} \label{Striinter2}
 \left\|\int_{0}^{t} D^{-1} \tilde{D}^{1 -\sigma} e^{\pm i (t-s) D} F(\cdot,s) \,ds \right\|_{L^{p} L^{q}(\Rm \times \Hm^n)} \lesssim \|F\|_{L^1 L^2(\Rm \times \Hm^n)}.
\end{equation}
By the identities 
\begin{align*}
 u(\cdot,t) & = \cos (t D) u_0 + \frac{\sin (t D)}{D} u_1 + \int_0^t \frac{\sin (t-s) D}{D} F(\cdot, s)\, ds;\\
 \partial_t u (\cdot,t) & = -D \sin (t D) u_0 + \cos (t D) u_1 + \int_0^t \left[\cos (t-s) D\right] F(\cdot, s)\, ds;
\end{align*}
we can combine the estimates (\ref{leftestimate}), (\ref{Striinter2}) and Lemma \ref{L1L2striinter} to finish the proof.
\end{proof}
\noindent If we choose $\sigma = \frac{n+1}{2} (\frac{1}{2} - \frac{1}{q})$ in the Theorem \ref{StriH01a} and apply the Sobolev embedding, we obtain another version of Strichartz estimates.
\begin{theorem} \label{StriH01b}
Assume $n \geq 3$. If $(p,q)$ satisfies
\begin{align*}
 &\frac{1}{p}, \frac{1}{q} \in \left(0,\frac{1}{2}\right);&  & \frac{1}{p} + \frac{n}{q} \geq \frac{n}{2} -1;&
% &\frac{1}{q} \geq \frac{1}{2} - \frac{2}{n+1}.&
\end{align*}
then there exists a constant $C$, such that the solution $u$ to the linear shifted wave equation ($0\in I$)
\[
\left\{\begin{array}{l} \partial_t^2 u - (\Delta_{\Hm^n}+\rho^2) u = F(x,t), \,\,\,\, (x,t)\in \Hm^n \times I;\\
u |_{t=0} = u_0; \\
\partial_t u |_{t=0} = u_1\end{array}\right.
\]
satisfies
\begin{align*}
 \|u\|_{L^{p} L^{q}(I \times \Hm^n)}&
  + \|(u, \partial_t u)\|_{C(I; H^{0,1}\times L^2(\Hm^n))} \\
 & \leq C \left( \|(u_0, u_1)\|_{H^{0,1}\times L^2(\Hm^n)} + \|F\|_{L^1 L^2 (I\times \Hm^n)} \right).
\end{align*}
\end{theorem}
Here we attach two figures, in which the grey regions illustrate all possible pairs $(p,q)$ that satisfy the conditions in Theorem \ref{StriH01b}, for two different cases: dimension $3$ (Figure \ref{hyper4d3}) and higher dimensions (Figure \ref{hyper4d4}).The lighter grey regions represent the pairs allowed in Theorem \ref{strichartzpre}, while the darker grey regions show new ``admissible'' pairs, which are obtained via the Sobolev embedding.
\begin{figure}[h]
  \centering
  \includegraphics{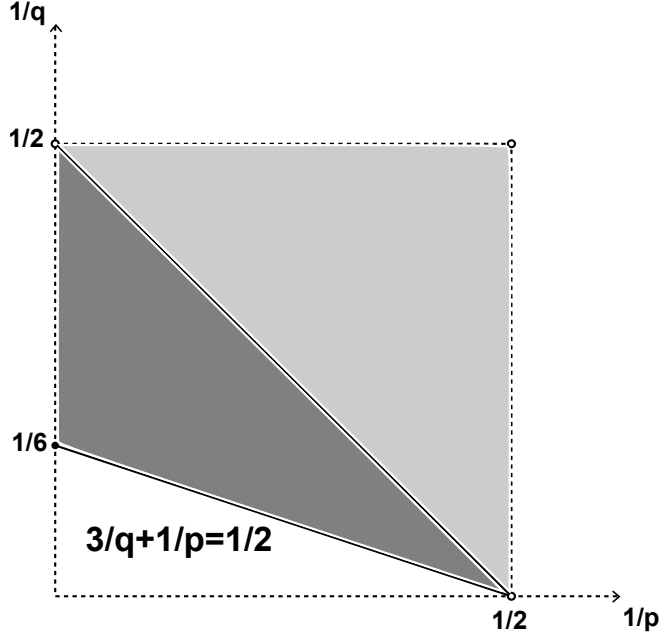}
  \caption{Admissible pairs $(p,q)$ in dimension 3}\label{hyper4d3}
\end{figure}
\begin{remark}
 Theorem \ref{StriH01b} also holds for the pair $(p,q) = (\infty, \frac{2n}{n-2})$ by the Sobolev embedding $H^{0,1}(\Hm^n) \hookrightarrow L^{2n/(n-2)}$ given in Proposition \ref{Sobolevem1}. Thus the pair $(\infty, \frac{2n}{n-2})$ is also marked as admissible in the figures.
\end{remark}
\begin{figure}[h]
  \centering
  \includegraphics{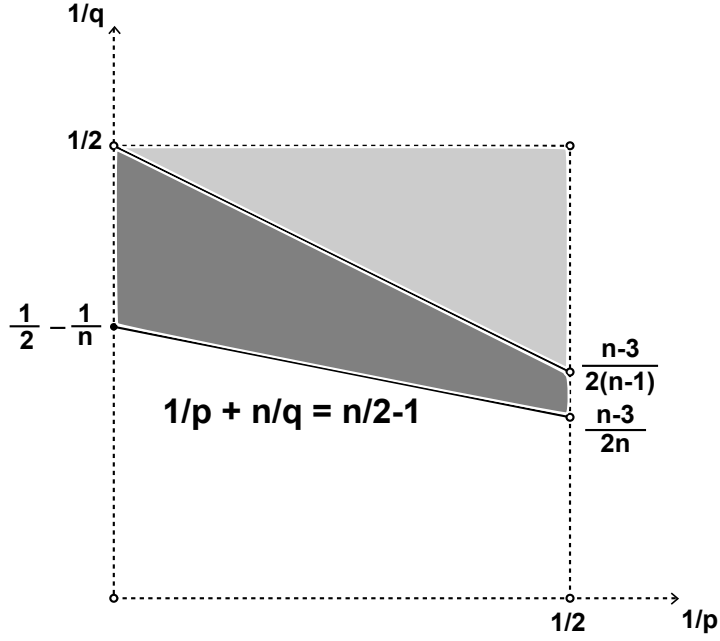}
  \caption{Admissible pairs $(p,q)$ in dimension 4 or higher}\label{hyper4d4}
\end{figure}
\section{Local Theory}

\begin{definition}
Assume $3 \leq n \leq 5$. We define the following space-time norm if $I$ is a time interval
\[
 \|u\|_{Y(I)} = \|u\|_{L^{\frac{n+2}{n-2}} L^{\frac{2(n+2)}{n-2}}(I \times \Hm^n)} = \|u\|_{L^{p_c} L^{2 p_c}(I \times \Hm^n)}.
\]
\end{definition}
\noindent Theorem \ref{StriH01b} claims that if $u$ is a solution to the linear equation $\partial_t^2 u - (\Delta_{\Hm^n} +\rho^2) u = F$ with initial data $(u_0,u_1)$, then we have
\[
 \|u\|_{Y(I)} + \|(u, \partial_t u)\|_{C(I;H^{0,1}\times L^2(\Hm^n))} \leq C \left(\|(u_0, u_1)\|_{H^{0,1}\times L^2(\Hm^n)} + \|F\|_{L^1 L^2 (I \times \Hm^n)}\right).
\]
Furthermore, a basic computation shows
\begin{align*}
 \|F(u)\|_{L^1 L^2 (I \times \Hm^n)} & \leq \|u\|_{Y(I)}^{p_c};\\
 \|F(u_1) - F(u_2)\|_{L^1 L^2 (I \times \Hm^n)} & \leq c_n \|u_1 -u_2\|_{Y(I)} \left(\|u_1\|_{Y(I)}^{p_c-1} + \|u_2\|_{Y(I)}^{p_c-1}\right).
\end{align*}
Combining these estimates with a fixed-point argument, we obtain the following local theory. (Our argument is standard, see for instance, \cite{bahouri, locad1, kenig, kenig1, ls, local1, ss2} for more details.)
\begin{definition}[Local solution] \label{localtheory}
Assume $3 \leq n \leq 5$. We say $u(t)$ is a solution of the equation (CP1) in a time interval $I$, if $(u(\cdot, t),\partial_t u(\cdot, t)) \in C(I;{H^{0, 1}}\times{L^2}(\Hm^n))$, with a finite norm $\|u\|_{Y(J)}$ for any bounded closed interval $J \subseteq I$ so that the integral equation
\[
 u(\cdot, t) = \mathbf{S}_{L,0} (t)(u_0,u_1) + \int_0^t \frac{\sin [(t-\tau)D]} {D} F(u(\cdot, \tau))\, d\tau
\]
holds for all time $t \in I$. 
\end{definition}
\begin{theorem} [Unique existence] 
For any initial data $(u_0,u_1) \in {H^{0, 1}}\times{L^2}(\Hm^n)$, there is a maximal interval $(-T_{-}(u_0,u_1), T_{+}(u_0,u_1))$ in which the equation (CP1) has a unique solution.
\end{theorem}
%\begin{proposition} [Uniqueness] 
%If there is another solution $\tilde{u}$ with the same initial data in the time interval $J$, then we have $J \subseteq (-T_-,T_+)$ and $u=\tilde{u}$ in $J$.
%\end{proposition}
\begin{proposition} [Scattering with small data]
There exists a constant $\delta_1 > 0$ such that if $\|(u_0,u_1)\|_{{H^{0, 1}}\times{L^2}(\Hm^n)} < \delta_1$, then the Cauchy problem (CP1) has a solution $u$ defined for all $t \in \Rm$ with $\|u\|_{Y(\Rm)} \lesssim \|(u_0,u_1)\|_{{H^{0, 1}}\times{L^2}(\Hm^n)}$.
\end{proposition}
\begin{proposition} [Standard finite time blow-up criterion] \label{standard blow up}
 If $T_{+} < \infty$, then $\|u\|_{Y([0,T_+))} = \infty$. Similarly if $T_-< \infty$, then $\|u\|_{Y((-T_-,0])} = \infty$. 
\end{proposition}
\begin{proposition} [Finite $Y$ norm implies scattering] \label{Yscattering} 
Let $u$ be a solution to (CP1). If $\|u\|_{Y([0,T_+))} < \infty$, then $T_+ = \infty$ and there exists a pair $(u_0^+,u_1^+) \in {H^{0, 1}}\times{L^2}(\Hm^n)$, such that
\[
   \lim_{t \rightarrow + \infty} \left\|(u(\cdot,t), \partial_t u(\cdot,t)) - \mathbf{S}_L (t) (u_0^+, u_1^+)\right\|_{{H^{0, 1}}\times{L^2}(\Hm^n)} = 0.
\]
A similar result holds in the negative time direction as well. 
\end{proposition}
\begin{theorem}  [Long-time perturbation theory] 
(See also \cite{kenig1, shen2}) Let $M$ be a positive constant. There exists a constant $\eps_0 = \eps_0 (M)>0$, such that if $\eps < \eps_0$, then for any approximation solution $\tilde{u}$ defined on $\Hm^n \times I$ ($0\in I$) and any initial data $(u_0,u_1) \in {H^{0, 1}}\times{L^2}(\Hm^n)$ satisfying
 \begin{align*}
           &\partial_t^2 \tilde{u} - (\Delta_{\Hm^n} + \rho^2) \tilde{u}=  F(\tilde{u}) + e(x,t), \qquad (x,t) \in \Hm^n \times I; \\
           &\|\tilde{u}\|_{Y(I)} < M; \qquad \|(\tilde{u}(\cdot, 0),\partial_t\tilde{u}(\cdot, 0))\|_{{H^{0, 1}}\times{L^2}(\Hm^n)}< \infty;\\
           &\|e(x,t)\|_{L^{1} L^{2}(I \times \Hm^n)}+ \|\mathbf{S}_{L,0} (t)(u_0-\tilde{u}(\cdot, 0),u_1 - \partial_t \tilde{u}(\cdot, 0))\|_{Y(I)} \leq \eps;
\end{align*}
there exists a solution $u(x,t)$ of (CP1) defined in the interval $I$ with the given initial data $(u_0,u_1)$ and satisfying
\[
    \|u(x,t) - \tilde{u}(x,t)\|_{Y(I)}  \leq C(M) \eps;
\]
\[
     \sup_{t \in I} \left\|\begin{pmatrix} u(\cdot, t)\\ \partial_t u(\cdot, t)\end{pmatrix} - \begin{pmatrix} \tilde{u}(\cdot, t)\\ \partial_t \tilde{u}(\cdot, t)\end{pmatrix}
            - \mathbf{S}_L (t) \begin{pmatrix} u_0 - \tilde{u}(\cdot, 0)\\ u_1 -\partial_t \tilde{u}(\cdot, 0)\end{pmatrix} \right\|_{{H^{0, 1}}\times{L^2}(\Hm^n)}\leq  C(M)\eps.
\]
\end{theorem}

\section{A Second Morawetz Inequality}

In my recent joint work with Staffilani \cite{subhyper}, we proved a Morawetz-type inequality 
\[
 \int_{-T_-}^{T_+} \int_{\Hm^n} |u|^{p+1} d\mu dt < \frac{4(p+1)}{p-1} \EE,
\]
if $u$ is a solution to the energy sub-critical, defocusing, semi-linear shifted wave equation $\partial_t^2 u - (\Delta_{\Hm^n}+\rho^2) u = - |u|^{p-1} u$ on $\Hm^n$. The main idea is to choose a suitable function $a$ and then apply the informal computation 
\begin{align*}
   & -\frac{d}{dt} \int_{\Hm^n} \partial_t u \cdot \left(\Db^\alpha a \Db_\alpha u + u \cdot \frac{\Delta a}{2} \right) d\mu \\
 = & \int_{\Hm^n} \left(\Db_\beta u \Db^\beta \Db^\alpha a \Db_\alpha u \right) d\mu
 - \frac{1}{4} \int_{\Hm^n} \left( |u|^2 \Delta \Delta a\right) d\mu
 + \frac{p-1}{2(p+1)} \int_{\Hm^n} \left( |u|^{p+1}\Delta a \right) d\mu
\end{align*}
on a solution $u$. In this section we prove a second and stronger Morawetz inequality by choosing a different function $a(r) = r$ and applying the same informal computation. The calculation turns out to be a little more complicated since the singularity of $r$ at the origin make it necessary to apply a smooth cut-off technique at this point. 

\begin{theorem} \label{Morawetz1A}
 Let $3 \leq n \leq 5$ and $(u_0,u_1)\in H^{0,1} \times L^2(\Hm^n)$ be initial data. Assume $u$ is the solution of (CP1) in the defocusing case with initial data $(u_0,u_1)$. Then the energy
\[
 \EE = \frac{1}{2}\|u(\cdot,t)\|_{H^{0,1}(\Hm^n)}^2 + \frac{1}{2}\|\partial_t u(\cdot,t)\|_{L^2(\Hm^n)}^2 + \frac{1}{p_c +1}\|u(\cdot,t)\|_{L^{p_c +1}(\Hm^n)}^{p_c +1}
\]
is a constant in the maximal lifespan $(-T_-,T_+)$. In addition, we have a Morawetz-type inequality
\[
 \int_{-T_-}^{T_+} \int_{\Hm^n} \frac{\rho (\cosh |x|) |u(x,t)|^{2n/(n-2)}}{\sinh |x|} d\mu(x) dt \leq n \EE.
\]
\end{theorem}

\begin{remark}
Throughout this section we will only consider real-valued solutions for convenience. Complex-valued solutions can be handled in the same manner.
\end{remark}

\begin{remark}
It suffices to prove Theorem \ref{Morawetz1A} with an additional assumption $u_0 \in H^1(\Hm^n)$. This is a consequence of the standard approximation techniques.  Given any initial data $(u_0,u_1) \in \dot{H}^{0,1}\times L^2(\Hm^n)$, we can find a sequence $(u_{0,n}, u_{1,n}) \in H^1 \times L^2(\Hm^n)$, such that 
\[
 \|(u_{0,n},u_{1,n})-(u_0,u_1)\|_{\dot{H}^{0,1} \times L^2} \rightarrow 0 \; \Rightarrow \; \EE(u_{0,n}, u_{1,n}) \rightarrow \EE(u_0,u_1). 
\]
Let $u$ and $\{u_n\}_{n \in {\mathbb Z}^+}$ be the corresponding solutions to (CP1) with these initial data. According to the perturbation theory we have
\begin{align*}
 &\|u_n - u\|_{Y(J)} \rightarrow 0;& &\|(u(\cdot,t), \partial_t u(\cdot,t))-(u_n(\cdot,t), \partial_t u_n(\cdot,t))\|_{C(J;\dot{H}^{0,1}\times L^2)} \rightarrow 0&
\end{align*}
for any closed bounded interval $J=[-T_1,T_2]$ contained in the maximal lifespan of $u$. A limiting process $n \rightarrow \infty$ shows that the energy conservation law and the Morawetz inequality hold for $u$ as long as they hold for the solutions $\{u_n\}_{n \in {\mathbb Z}^+}$. 
\end{remark}

In the rest of the section we always assume that $u$ is a solution to (CP1) with initial data $(u_0,u_1) \in H^1 \times L^2 (\Hm^n)$.

\subsection{Preliminary Results}
\begin{lemma} \label{H1L2}
 We have $(u, \partial_t u) \in C((-T_-,T_+); H^1 \times L^2 (\Hm^n))$.
\end{lemma}
\begin{proof}
 We have already known $(u, \partial_t u) \in C((-T_-,T_+); H^{0,1} \times L^2 (\Hm^n))$ by Definition \ref{localtheory}. Therefore it is sufficient to show $u \in C((-T_-,T_+); L^2)$. This is clearly true since $\partial_t u \in C((-T_-,T_+); L^2)$ and $u_0 \in L^2$.
\end{proof}
\begin{lemma} \label{def of a}
  Assume $n \geq 3$. Let $(r,\Theta)$ be the polar coordinates on $\Hm^n$. Then the function $a(r) = r$ is smooth in $\Hm^n$ except for $r=0$ and satisfies
\begin{align*}
 &|\nabla a (r)| = 1;& &\Db^2 a \geq 0;&\\
 &\Delta_{\Hm^n} a = \frac{2\rho \cosh r}{\sinh r};& &\partial_r \Delta_{\Hm^n} a = \frac{-2\rho}{\sinh^2 r};&\\
 &\partial_{r}^2 \Delta_{\Hm^n} a = \frac{4 \rho \cosh r}{\sinh^3 r};& &\Delta_{\Hm^n} \Delta_{\Hm^n} a = \frac{4\rho (1-\rho) \cosh r}{\sinh^3 r} \leq 0. &
\end{align*}
\end{lemma}

The proof follows an explicit calculation, thus we omit the details. In fact, the inequality $\Db^2 a \geq 0$ is a well-known consequence of the fact that the hyperbolic space $\Hm^n$ has a negative constant sectional curvature. The following formula also helps in the calculation regarding the Laplace operator $\Delta_{\Hm^n}$
\[
 \Delta_{\Hm^n} = \frac{\partial^2}{\partial r^2} + \frac{2\rho \cosh r}{\sinh r}\cdot \frac{\partial}{\partial r} + \frac{1}{\sinh^2 r} \cdot \Delta_{{\mathbb S}^{n-1}}.
\]

\begin{definition} \label{smoothcutoff}
 Let $\psi: [0,\infty) \rightarrow [0,1]$ be a non-increasing smooth cut-off function satisfying
\[
 \psi(r) = \left\{\begin{array}{ll}
  1, & r < 1;\\
  0, & r > 2.
 \end{array}\right.
\]
If $\delta \in (0,1/10]$, we define a radial cut-off function $\psi_{\delta}(r) = \psi (\delta r) (1 - \psi(r/\delta))$ on $\Hm^n$. It is clear that
\[
 |\nabla \psi_\delta (r,\Theta)| \lesssim \left\{\begin{array}{ll}\delta, & 1/\delta < r < 2/\delta;\\
 1/\delta, & \delta < r< 2\delta;\\
 0, & \hbox{otherwise}. \end{array}\right.
\]
\end{definition}
\begin{lemma}\label{smoothing} (See Lemma 4.11 in \cite{subhyper}, also \cite{staten} for more general case) Let $\mathbf{\tilde{P}}_\eps$ be the smoothing operator defined by the Fourier multiplier $\lambda \rightarrow e^{-\eps^2 \lambda^2}$. Given any $2 \leq q < \infty$, we have $\|\mathbf{\tilde{P}}_\eps\|_{L^q(\Hm^n) \rightarrow L^q(\Hm^n)} \leq C_q < \infty$ for any $\eps \in \left(0,\frac{1}{10}\right)$. Furthermore, if $v \in L^q (\Hm^n)$, then $\|v - \mathbf{\tilde{P}}_\eps v\|_{L^q} \rightarrow 0$ as $\eps \rightarrow 0$.
\end{lemma}
\paragraph{Space-time smoothing operator} Let $\phi(t)$ be a smooth, nonnegative, even function compactly supported in $[-1,1]$ with $\int_{-1}^{1} \phi (t) dt =1$. Given a closed interval $[-T_1,T_2] \subset (-T_-,T_+)$, we define $u_\eps$ and $F_\eps$ as the smooth version of $u$ and $F$ if $\eps < \eps_0 = (1/2)\min\{1/10, T_+ - T_2, T_{-} -T_1\}$.
\begin{equation}
 u_\eps (t) = \int_{-1}^{+1} \phi(s) \tilde{P}_\eps u(\cdot,t+s\eps) ds;\quad
 F_\eps (t) = \int_{-1}^{+1} \phi(s) \tilde{P}_\eps F(u(\cdot,t+s\eps)) ds. \label{smoothdef}
\end{equation}
The function $u_\eps$ is a smooth solution to the shifted wave equation
\begin{equation} \label{app linear}
 \partial_t^2 u_\eps - (\Delta_{\Hm^n}+ \rho^2) u_\eps = F_{\eps}
\end{equation}
in the time interval $[-T_1,T_2]$. Combining the fact $\|u\|_{Y([-T_1-\eps_0, T_2 + \eps_0])}<\infty$, the inequality
\begin{align*}
 \|F(u_\eps) - F_\eps\|_{L^1 L^2(I)} \leq & \|F(u_\eps) - F(u)\|_{L^1 L^2(I)} + \|F(u) - F_\eps\|_{L^1 L^2(I)}\\
 \leq & C \|u_\eps - u\|_{Y(I)} \left(\|u_\eps\|_{Y(I)}^{p_c-1} + \|u\|_{Y(I)}^{p_c -1}\right) + \|F(u) - F_\eps\|_{L^1 L^2(I)},
\end{align*}
with Lemma \ref{H1L2} and Lemma \ref{smoothing}, we immediately have 

\begin{lemma} \label{convergenceueps} Let $t_0$ be an arbitrary time in $[-T_1,T_2]$. The functions $u_\eps$ and $F_\eps$ satisfy
\begin{align*}
  &\lim_{\eps \rightarrow 0} \|F(u_\eps) - F_\eps \|_{L^1 L^2 ([-T_1,T_2]\times \Hm^n)} = 0;\\
  &\lim_{\eps \rightarrow 0} \|(u_\eps (\cdot, t_0), \partial_t u_\eps (\cdot, t_0))- (u (\cdot, t_0),\partial_t u(\cdot, t_0))\|_{H^1 \times L^2 (\Hm^n)} =0;\\
  &\lim_{\eps \rightarrow 0} \|u_\eps (\cdot, t_0)- u (\cdot, t_0)\|_{L^{p_c+1} (\Hm^n)} =0;\\
  &M_1 := \sup_{\eps < \eps_0} \|(u_\eps, \partial_t u_\eps)\|_{C([-T_1,T_2]; H^1\times L^2 (\Hm^n))} < \infty.
\end{align*}
The third line is a combination of the Sobolev embedding $H^1 \hookrightarrow L^{p_c+1}$ and the second line.
\end{lemma}
\subsection{Energy Conservation Law}
\begin{proposition} [Energy Conservation Law] The energy
\[
 \EE(t) = \frac{1}{2}\|u(\cdot,t)\|_{H^{0,1}(\Hm^n)}^2 + \frac{1}{2}\|\partial_t u(\cdot,t)\|_{L^2(\Hm^n)}^2 + \frac{1}{p_c +1}\|u(\cdot,t)\|_{L^{p_c +1}(\Hm^n)}^{p_c +1}
\]
is a (finite) constant independent of $t \in (-T_-, T_+)$.
\end{proposition}
\begin{proof}
Without loss of generality, let us assume $t_0 \in (0,T_+)$. We can choose a time interval $[-T_1,T_2] \subset (-T_-, T_+)$ so that $t_0 <T_2$, smooth out the solution $u$ as in (\ref{smoothdef}) and define
\[
 \EE_{\eps, \delta}(t) = \int_{\Hm^n} \left(\frac{1}{2} |\nabla u_\eps (x,t)|^2 - \frac{\rho^2}{2} |u_\eps (x,t)|^2 + \frac{1}{2} |\partial_t u_\eps (x,t)|^2 + \frac{1}{p_c +1} |u_\eps (x,t)|^{p_c +1}\right)\psi_\delta d\mu(x).
\]
We differentiate in $t$ and obtain 
\[
 \EE_{\eps,\delta}'(t) =  \int_{\Hm^n} \left[F_\eps - F(u_\eps) \right](\partial_t u_\eps) \psi_\delta d\mu
 - \int_{\Hm^n} (\Db^\alpha \psi_{\delta} \Db_\alpha u_\eps) (\partial_t u_\eps) d\mu
\]
Here we need to use the fact that $u_\eps$ solves \eqref{app linear} and follow the same calculation we carried on in Section 4.2 of \cite{subhyper}.  A basic integration shows 
\begin{align*}
 \left|\EE_{\eps,\delta}(t_0) - \EE_{\eps,\delta}(0)\right| \leq & \int_0^{t_0}\! \int_{\Hm^n} \left|F_\eps - F(u_\eps) \right||\partial_t u_\eps| \psi_\delta d\mu dt
+\int_0^{t_0}\! \int_{\Hm^n} \left|(\Db^\alpha \psi_{\delta} \Db_\alpha u_\eps ) \partial_t u_\eps\right| d\mu dt\\
\lesssim & \|F_\eps - F(u_\eps)\|_{L^1 L^2} \|\partial_t u_\eps\|_{L^\infty L^2} + \delta t_0 \|u_\eps\|_{L^\infty ([0,t_0]; H^1(\Hm^n))} \|\partial_t u_\eps\|_{L^\infty L^2}\\
& + \delta^{n-1} t_0 \|\nabla u_\eps\|_{L^\infty (B(\mathbf{0},1)\times [0,t_0])} \|\partial_t u_\eps\|_{L^\infty (B(\mathbf{0},1)\times [0,t_0])}.
\end{align*}
If we send first $\delta \rightarrow 0$, then $\eps \rightarrow 0$ and apply Lemma \ref{convergenceueps}, we obtain the energy conservation law.
\end{proof}
\begin{remark} \label{invariance of energy focus}
 The energy of a solution to (CP1) in the focusing case is also a constant under the same assumptions, because the defocusing assumption has not been used in the argument above.
\end{remark}

\subsection{Proof for the Morawetz inequality}
Since the argument is similar to the one we carried on in Section 4 of \cite{subhyper}, we will give an outline of the proof only. Given an arbitrary time interval $[-T_1,T_2] \subset (-T_-,T_+)$, we can smooth out $u$, the non-linear term $F(u)$ as in (\ref{smoothdef}) and define
\[
 M_{\eps, \delta} (t) = -\int_{\Hm^n}  \partial_t u_\eps (x,t) \left(\Db^\alpha a \Db_\alpha u_\eps (x,t) + u_\eps(x,t) \cdot \frac{\Delta a}{2} \right)\psi_\delta d\mu(x)
\]
for $t\in [-T_1,T_2]$. Here we choose the function $a = r$ whose properties have been given in Lemma \ref{def of a} and the smooth cut-off function $\psi_{\delta}$ in Definition \ref{smoothcutoff}. We first combine a basic differentiation in $t$, integration by parts, the facts $\Db_\beta \psi_\delta \Db^\beta \Delta a \leq 0$ in $B(\mathbf{0},1)$ and $\mathbf{D}^2 a \geq 0$ to obtain
\begin{align}
 M_{\eps, \delta}'(t) \geq &\;\frac{1}{n} \int_{\Hm^n} |u_\eps|^{p_c +1}\Delta a \psi_\delta \,d\mu + \frac{1}{4} \int_{\Hm^n} (-\Delta \Delta a) |u_\eps|^2 \psi_\delta\,  d\mu  - e(\eps, \delta)\nonumber\\
 & \qquad - \int_{\Hm^n} \left[F_\eps - F(u_\eps)\right] \left(\Db^\alpha a \Db_\alpha u_\eps + u_\eps \cdot \frac{\Delta a}{2} \right)\psi_\delta d\mu. \label{derivative of M}
\end{align}
Throughout the proof we use the notation $e(\eps,\delta)$ to represent (possibly different) error terms satisfying
\begin{align*}
 |e(\eps,\delta)| \lesssim & \delta (M_1^2 + M_1^{p_c +1}) + \delta^{n-2} \|(u_\eps, \nabla u_\eps)\|_{C (\bar{B}(\mathbf{0},1)\times [-T_1,T_2])}^2 \\
 & \qquad + \delta^{n-1} \|u_\eps\|_{C (\bar{B}(\mathbf{0},1)\times [-T_1,T_2])}^{p_c+1} + \delta^{n-1} \|\partial_t u_\eps\|_{C (\bar{B}(\mathbf{0},1)\times [-T_1,T_2])}^2.
\end{align*}
Next we apply integration by parts again and make an estimate for an arbitrary $t \in [-T_1,T_2]$. 
\begin{align}
 & \left\|\Db^\alpha a \Db_\alpha u_\eps (\cdot, t) + u_\eps (\cdot, t)\cdot \frac{\Delta a}{2}\right\|_{L^2(\Hm^n;\psi_\delta d\mu)}^2\nonumber\\
 % = & \int_{\Hm^n} \left(\Db^\alpha a \Db_\alpha u_\eps + u_\eps \cdot \frac{\Delta a}{2}\right)^2 \psi_\delta d\mu\nonumber\\
%  = & \int_{\Hm^n} \left(|\Db^\alpha a \Db_\alpha u_\eps|^2 + \frac{1}{4}|u_\eps|^2(\Delta a)^2+ \frac{\Delta a}{2}  \Db^\alpha a \Db_\alpha (u_\eps^2) \right)\psi_\delta d\mu\nonumber\\
 % \leq & \int_{\Hm^n} \left(|\nabla a|^2 |\nabla u_\eps|^2 - \frac{1}{4}|u_\eps|^2(\Delta a)^2  - \frac{1}{2} |u_\eps|^2 \Db^\alpha a \Db_\alpha (\Delta a)\right)\psi_\delta d\mu\nonumber \\
 % &\quad - \frac{1}{2} \int_{\Hm^n} (\Db_\alpha \psi_\delta \Db^\alpha a) |u_\eps|^2 \Delta a\, d\mu \nonumber\\
  \leq & \int_{\Hm^n} \left(|\nabla u_\eps|^2 - \left(\rho^2 + \frac{\rho^2 -\rho}{\sinh^2 |x|}\right) |u_\eps|^2 \right)\psi_\delta d\mu + e(\eps,\delta)\nonumber\\
  \leq & \int_{\Hm^n} \left(|\nabla u_\eps|^2 - \rho^2 |u_\eps|^2 \right)\psi_\delta d\mu + e(\eps,\delta) \leq M_1^2 + e(\eps,\delta). \label{estimateb1}
\end{align}
As a result we have
\begin{align*}
  & - \int_{-T_1}^{T_2} \int_{\Hm^n} \left[F_\eps - F(u_\eps)\right] \left(\Db^\alpha a \Db_\alpha u_\eps + u_\eps \cdot \frac{\Delta a}{2} \right)\psi_\delta d\mu dt\\
  \geq & - \left\|F_\eps - F(u_\eps)\right\|_{L^1 L^2 ([-T_1,T_2]\times \Hm^n)} \left\|\Db^\alpha a \Db_\alpha u_\eps + u_\eps \cdot \frac{\Delta a}{2} \right\|_{L^\infty ([-T_1,T_2]; L^2(\Hm^n; \psi_\delta d\mu))}\\
  \geq & - (M_1^2 + e(\eps,\delta))^{1/2} \left\|F_\eps - F(u_\eps)\right\|_{L^1 L^2 ([-T_1,T_2]\times \Hm^n)} .
\end{align*}
We substitute the last integral in \eqref{derivative of M} by the lower bound above, integrate both sides and obtain an inequality 
\begin{align}
 M_{\eps, \delta} (T_2) -&  M_{\eps, \delta} (-T_1) \geq  \frac{1}{n} \int_{-T_1}^{T_2} \int_{\Hm^n} |u_\eps|^{p_c +1} \Delta a \psi_\delta d\mu dt + \frac{1}{4} \int_{-T_1}^{T_2} \int_{\Hm^n} (-\Delta \Delta a) |u_\eps|^2 \psi_\delta  d\mu dt \nonumber\\
 & -  (M_1^2 + e(\eps,\delta))^{1/2} \left\|F_\eps - F(u_\eps)\right\|_{L^1 L^2 ([-T_1,T_2]\times \Hm^n)} - [T_2+T_1] e(\eps, \delta). \label{inequality1}
\end{align}
On the other hand, we can find an upper bound of $M(t_0)$ for each $t_0 \in [-T_1,T_2]$ by (\ref{estimateb1}).
\begin{align*}
|M(t_0)| =& \left| \int_{\Hm^n} \partial_t u_\eps (x,t_0) \left( \Db^\alpha a \Db_\alpha u_\eps (x,t_0) + \frac{1}{2} u_\eps (x,t_0) \Delta a \right) \psi_\delta d\mu(x) \right|\\
 \leq &\frac{1}{2} \int_{\Hm^n} \left( |\partial_t u_\eps|^2 + \left(\Db^\alpha a \Db_\alpha u_\eps + \frac{1}{2}u_\eps \Delta a \right)^2 \right) \psi_\delta d\mu\\
 \leq &\frac{1}{2} \int_{\Hm^n} \left(|\partial_t u_\eps|^2 + |\nabla u_\eps|^2 - \rho^2 |u_\eps|^2 \right)\psi_\delta d\mu + e(\eps,\delta).%\\
 %= & E_0 (u_\eps(t),\partial_t u_\eps(t)) + e(\eps,\delta).
\end{align*}
We combine this with the inequality (\ref{inequality1}), substitute $\Delta a$, $\Delta \Delta a$ by their specific expressions by Lemma \ref{def of a} and then let $\delta \rightarrow 0$ to obtain 
\begin{align*}
 & \frac{1}{n} \int_{-T_1}^{T_2} \int_{\Hm^n} \frac{2\rho (\cosh |x|) |u_\eps|^{p_c +1}}{\sinh |x|} d\mu dt + \int_{-T_1}^{T_2} \int_{\Hm^n} \frac{\rho (\rho-1)(\cosh |x|)^3 |u_\eps|^2}{\sinh |x|} d\mu dt\\
  \leq & \EE_0(u_\eps(\cdot, -T_1),\partial_t u_\eps(\cdot, -T_1)) + \EE_0(u_\eps(\cdot, T_2), \partial_t u_\eps (\cdot, T_2))
  +  M_1 \left\|F_\eps - F(u_\eps)\right\|_{L^1 L^2 ([-T_1,T_2]\times \Hm^n)}.
\end{align*}
Here $\EE_0$ is the energy of the linear shifted wave equation on $\Hm^n$ defined as
\[
 \EE_0(v_0,v_1) = \frac{1}{2} \int_{\Hm^n} \left(|v_1|^2 + |\nabla v_0|^2 - \rho^2 |v_0|^2 \right) d\mu = \frac{1}{2}\|v_0\|_{H^{0,1}(\Hm^n)}^2 + \frac{1}{2} \|v_1\|_{L^2 (\Hm^n)}^2.
\]
Sending $\eps$ to zero gives
\begin{align*}
 & \frac{1}{n} \int_{-T_1}^{T_2} \int_{\Hm^n} \frac{2\rho (\cosh |x|) |u|^{p_c +1}}{\sinh |x|} d\mu dt + \int_{-T_1}^{T_2} \int_{\Hm^n} \frac{\rho (\rho-1)(\cosh |x|)^3 |u|^2}{\sinh |x|} d\mu dt\\
 &\qquad \leq \EE_0(u(-T_1),\partial_t u(-T_1)) + \EE_0(u(T_2), \partial_t u (T_2)) \leq 2 \EE(u_0,u_1).
\end{align*}
Since the argument above is valid for any time interval $[-T_1,T_2]$ satisfying $-T_- < -T_1 < 0 < T_2 <T_+$, we can finally finish the proof by letting $T_1 \rightarrow T_-$ and $T_2 \rightarrow T_+$.

\section{Scattering Results with Radial Initial Data}
\begin{theorem} \label{scattering1}
Assume $3 \leq n \leq 5$. Let $(u_0,u_1)$ be a pair of radial initial data in $H^{0,1}\times L^2 (\Hm^n)$. Then the solution $u$ to (CP1) in the defocusing case with initial data $(u_0,u_1)$ exists globally in time and scatters. More precisely, the solution $u$ satisfies
\begin{itemize}
  \item Its maximal lifespan $(-T_-(u_0,u_1), T_+(u_0,u_1)) = \Rm$.
  \item There exist two pairs $(u_0^\pm, u_1^\pm) \in H^{0,1}\times L^2 (\Hm^n)$ such that
       \[
          \lim_{t \rightarrow \pm \infty} \left\|\left(u(\cdot, t), \partial_t u(\cdot,t)\right)- \mathbf{S}_L (t)(u_0^\pm, u_1^\pm)\right\|_{{H^{0, 1}}\times{L^2}(\Hm^n)} = 0.
       \]
\end{itemize}
\end{theorem}
\begin{proof}
First of all, the Morawetz inequality (see Theorem \ref{Morawetz1A}) reads
\begin{equation} \label{Morawetz001}
 \int_{-T_-}^{T_+} \int_{\Hm^n} \frac{\rho (\cosh |x|) |u(x,t)|^{2n/(n-2)}}{\sinh |x|} d\mu(x) dt \leq n \EE(u_0,u_1) < \infty.
\end{equation}
According to Lemma \ref{radialest} and the energy conservation law we obtain a point-wise estimate 
\begin{align*}
 |u(x,t)| & \lesssim \frac{|x|^{1/2}}{(\sinh |x|)^\rho} \|u(\cdot,t)\|_{H^{0,1}(\Hm^n)} \lesssim \left(\frac{\rho \cosh |x|}{\sinh |x|}\right)^{\rho -1/2} \|u(\cdot, t)\|_{H^{0,1}(\Hm^n)}\\
 & \lesssim \left(\frac{\rho \cosh |x|}{\sinh |x|}\right)^{\rho -1/2} \left[\EE(u_0,u_1)\right]^{1/2}.
\end{align*}
As a result, the inequality (\ref{Morawetz001}) implies 
\[
 \int_{-T_-}^{T_+} \int_{\Hm^n} |u|^{\frac{2n}{n-2} + \frac{1}{\rho-1/2}} d\mu dt \lesssim \left[\EE(u_0,u_1)\right]^{1 + \frac{1/2}{\rho-1/2}}.
\]
A basic calculation shows $2n/(n-2) + 1/(\rho-1/2) = 2(n+1)/(n-2)$. Thus we have
\begin{equation} \label{LLbound}
 u \in L^{\frac{2(n+1)}{n-2}} L^{\frac{2(n+1)}{n-2}} ((-T_-,T_+)\times \Hm^n).
\end{equation}
For a small positive number $\kappa$, we define a pair $(p,q)$ by
\begin{equation} \label{def of pq}
 (p_c - \kappa) \left(\frac{1}{p}, \frac{1}{q}\right) + \kappa\left(\frac{n-2}{2(n+1)}, \frac{n-2}{2(n+1)} \right)= \left(1, \frac{1}{2}\right).
\end{equation}
It is clear that $(p,q) \rightarrow (p_c, 2 p_c)$ as $\kappa \rightarrow 0^+$. As a result, if we fix $\kappa$ to be a sufficiently small positive number, we have $1/p, 1/q \in (0,1/2)$. The definition (\ref{def of pq}) also guarantees that $1/p +n/q = n/2 -1$. We apply Strichartz estimates (Theorem \ref{StriH01b}), the energy conservation law, as well as an interpolation between $L^p L^q$ and $L^{\frac{2(n+1)}{n-2}} L^{\frac{2(n+1)}{n-2}}$ norms to obtain
\begin{align}
 \|u\|_{Y([a,b])} + &\|u\|_{L^p L^q ([a,b]\times \Hm^n)} \nonumber \\ 
 \leq & C \|(u(\cdot, a), \partial_t u(\cdot, a))\|_{H^{0,1} \times L^2(\Hm^n)} + C \|F(u)\|_{L^1 L^2 ([a,b]\times \Hm^n)} \nonumber \\
 \leq & C (2 \EE(u_0,u_1))^{1/2} + C \|u\|_{L^{\frac{2(n+1)}{n-2}} L^{\frac{2(n+1)}{n-2}}([a,b]\times \Hm^n)}^\kappa \|u\|_{L^p L^q ([a,b]\times \Hm^n)}^{p_c -\kappa}. \label{lin011}
\end{align}
Here $[a,b]$ may be any sub-interval of $(-T_-,T_+)$. Let us define $M = (2 \EE(u_0,u_1))^{1/2}$ and choose a small constant $\eta>0$ so that $2CM > CM + C\eta^\kappa (2CM)^{p_c -\kappa}$. According to the fact (\ref{LLbound}) we can fix a time $a \in (0,T_+)$ sufficiently close to $T_+$, so that the inequality
\[
 \|u\|_{L^{\frac{2(n+1)}{n-2}} L^{\frac{2(n+1)}{n-2}}([a,T_+)\times \Hm^n)} < \eta
\]
holds. Given any time $b \in [a,T_+)$, the inequality (\ref{lin011}) implies
\begin{align*}
  \|u\|_{Y([a,b])} + \|u\|_{L^p L^q ([a,b]\times \Hm^n)} \leq CM + C \eta^\kappa \left( \|u\|_{Y([a,b])} +  \|u\|_{L^p L^q ([a,b]\times \Hm^n)}\right)^{p_c -\kappa}.
\end{align*}
By a continuity argument in $b$ we obtain the following upper bound independent of $b \in [a,T_+)$. 
\[
  \|u\|_{Y([a,b])} + \|u\|_{L^p L^q ([a,b]\times \Hm^n)} < 2CM. 
\]
We make $b \rightarrow T_+$ and finally conclude that $\|u\|_{Y ([a,T_+))} \leq 2CM < \infty$. The global existence and scattering in the positive time direction immediately follows Proposition \ref{standard blow up} and Proposition \ref{Yscattering}. The other time direction can be handled in the same way since the shifted wave equation is time-reversible.
\end{proof}
\section*{Acknowledgements}
I would like to express my sincere gratitude to Professor Gigliola Staffilani for her helpful discussions and comments.

\end{document}